\documentclass[12pt]{article}
\usepackage{lineno,hyperref}
\modulolinenumbers[5]
\usepackage{color}

\usepackage{graphics,amsmath,amssymb}
\usepackage{amsthm}
\usepackage{amsfonts}
\usepackage{latexsym}
\usepackage{epsf}

\newcommand{\sstirling}[2]{\genfrac\{\}{0pt}{}{#1}{#2}}
\newcommand{\rbstirling}[2]{\genfrac\langle\rangle{0pt}{}{#1}{#2}}
\def\a{{\bf a}}

\theoremstyle{plain}
\newtheorem{theorem}{Theorem}
\newtheorem{corollary}[theorem]{Corollary}

\theoremstyle{definition}

\theoremstyle{remark}
\newtheorem{remark}[theorem]{Remark}

\begin{document}

\begin{center}
\vskip 1cm{\LARGE\bf Bivariate Extension of the $r$-Dowling Polynomials and the Generalized Spivey's Formula}
\vskip 1cm
\large    
Mahid M. Mangontarum\\
Department of Mathematics\\
Mindanao State University-Main Campus\\
Marawi City 9700\\
Philippines \\
\href{mailto:mmangontarum@yahoo.com}{\tt mmangontarum@yahoo.com} \\
\href{mailto:mangontarum.mahid@msumain.edu.ph}{\tt mangontarum.mahid@msumain.edu.ph}\end{center}

\vskip .2 in

\begin{abstract}
In this paper, we extend the $r$-Dowling polynomials to their bivariate forms. Several properties that generalize those of the bivariate Bell and $r$-Bell polynomials are established. Finally, we obtain two forms of generalized Spivey's formula.
\end{abstract}

\section{Introduction}

The Bell numbers $B_n$ are defined by the sum
\begin{equation}
B_n=\sum_{k=0}^n\sstirling{n}{k},\label{b1}
\end{equation}
where $\sstirling{n}{k}$ denote the Stirling numbers of the second kind, and are known to satisfy the recurrence relation given by
\begin{equation}
B_{n+1}=\sum_{k=0}^n\binom{n}{k}B_k.\label{b2}
\end{equation}
The numbers $\sstirling{n}{k}$ count the number of ways to partition a set $X$ of $n$ elements into $k$ non empty subsets. With this, it is obvious that $B_n$ count the total number of partitions of the set $X$. Using the same combinatorial interpretation, Spivey \cite{Spivey} obtained a generalized recurrence for $B_n$ which unifies \eqref{b1} and \eqref{b2}, viz.
\begin{equation}
B_{\ell+n}=\sum_{k=0}^{\ell}\sum_{i=0}^nk^{n-i}\binom{n}{i}\sstirling{\ell}{k}B_i.\label{spivey}
\end{equation}
The Bell polynomials, denoted by $B_n(x)$, are defined by
\begin{equation}
B_n(x)=\sum_{k=0}^n\sstirling{n}{k}x^k.\label{b1.1}
\end{equation}
Gould and Quaintance \cite{Gould} established the polynomial version of \eqref{spivey} as follows
\begin{equation}
B_{\ell+n}(x)=\sum_{k=0}^{\ell}\sum_{i=0}^nk^{n-i}\binom{n}{i}\sstirling{\ell}{k}B_i(x)x^j\label{xspivey}
\end{equation}
by means of generating functions. The same identity was also obtained by Belbachir and Mihoubi \cite[Theorem 1]{Belbachir} using a method that involve decomposition of the $B_n(x)$ into a certain polynomial basis and by Boyadzhiev \cite[Proposition 3.2]{Khristo} using the Mellin derivatives.

Recently, Zheng and Li \cite{Zheng} defined the bivariate Bell polynomials by
\begin{equation}
B_n(x,y)=\sum_{k=0}^n\sstirling{n}{k}(x)_ky^k,\label{b1.2}
\end{equation}
where $(x)_k=x(x-1)\cdots(x-k+1),\ (x)_0=1$, with the following exponential generating function \cite[Theorem 1]{Zheng}:
\begin{equation}
\sum_{n=0}^{\infty}B_n(x,y)\frac{t^n}{n!}=\left[1+y(e^t-1)\right]^x.\label{b3}
\end{equation}
With this notion, they were able to obtain the bivariate extension of Spivey's formula \cite[Theorem 2]{Zheng} given by
\begin{equation}
B_{\ell+n}(x,y)=\sum_{k=0}^{\ell}\sum_{i=0}^nk^{n-i}\binom{n}{i}\sstirling{\ell}{k}B_i(x-k,y)(x)_ky^k.\label{xyspivey}
\end{equation}
Equation \eqref{xspivey} can be recovered from this formula by replacing $y$ with $y/x$ and taking the limit as $x\rightarrow \infty$. The $r$-Stirling numbers of the second kind, denoted by $\sstirling{n}{k}_r$, are defined by Broder \cite{Bro} as the number of partitions of the $n$-element set $X$ into $k$ non empty disjoint subsets such that the elements $1,2,\ldots,r$ are in distinct subsets. These numbers are known to satisfy the horizontal generating function \cite[Theorem 22]{Bro}
\begin{equation}
(t+r)^n=\sum_{k=0}^n\sstirling{n+r}{k+r}_r(t)_k
\end{equation}
and the exponential generating function \cite[Theorem 16]{Bro}
\begin{equation}
\sum_{n=0}^{\infty}\sstirling{n+r}{k+r}_r\frac{t^n}{n!}=\frac{1}{k!}e^{rt}(e^t-1)^k.
\end{equation}
By defining the bivariate $r$-Bell polynomials by
\begin{equation}
B_{n,r}(x,y)=\sum_{k=0}^n\sstirling{n+r}{k+r}_r(x)_ky^k,\label{b4}
\end{equation}
Zheng and Li \cite{Zheng} were also able to obtain the following generalizations of \eqref{xyspivey}:
\begin{equation}
B_{\ell+n,r}(x,y)=\sum_{k=0}^{\ell}\sum_{i=0}^nk^{n-i}\binom{n}{i}\sstirling{\ell+r}{k+r}_rB_{i,r}(x-k,y)(x)_ky^k\label{xyspiveyr1}
\end{equation}
and
\begin{equation}
B_{\ell+n,r}(x,y)=\sum_{k=0}^{\ell}\sum_{i=0}^n(k+r)^{n-i}\binom{n}{i}\sstirling{\ell+r}{k+r}_rB_i(x-k,y)(x)_ky^k.\label{xyspiveyr2}
\end{equation}
Replacing $y$ with $1/x$ and taking the limit as $x\rightarrow \infty$ in these formulas give \cite[Corollaries 9 and 10]{Zheng}
\begin{equation}
B_{\ell+n,r}=\sum_{k=0}^{\ell}\sum_{i=0}^nk^{n-i}\binom{n}{i}\sstirling{\ell+r}{k+r}_rB_{i,r}\label{spiveyr1}
\end{equation}
and
\begin{equation}
B_{\ell+n,r}=\sum_{k=0}^{\ell}\sum_{i=0}^n(k+r)^{n-i}\binom{n}{i}\sstirling{\ell+r}{k+r}_rB_i.\label{spiveyr2}
\end{equation}
Equation \eqref{spiveyr2} is a generalization of \eqref{spivey} proved by Mez\H{o} \cite[Theorem 2]{Mez2} using the combinatorial interpretation of $\sstirling{n}{k}_r$ for the $r$-Bell numbers \cite[Equation 2]{Mez1}
\begin{equation}
B_{n,r}=\sum_{k=0}^n\sstirling{n+r}{k+r}_r.\label{rb1}
\end{equation}
Notice that we used the notation $\sstirling{\ell+r}{k+r}_r$ in the above equations instead of just $\sstirling{\ell}{k}_r$ for consistency. On the other hand, equation \eqref{spiveyr1} and its polynomial version \eqref{xyspiveyr1} appear in the paper of Mangontarum and Dibagulun \cite[Corollary 3]{Mahid} as particular case of the formula \cite[Theorem 2]{Mahid}
\begin{equation}
D_{m,r}(\ell+n;x)=\sum_{k=0}^{\ell}\sum_{i=0}^n(mk)^{n-i}W_{m,r}(n,k)\binom{n}{i}D_{m,r}(i;x)x^k,\label{DowlingSpivey1}
\end{equation}
where $D_{m,r}(\ell+n;x)$ denote the $r$-Dowling polynomials \cite{Cheon} defined by
\begin{equation}
D_{m,r}(n;x)=\sum_{k=0}^nW_{m,r}(n,k)x^k\label{rD1}
\end{equation}
and $W_{m,r}(n,k)$ denote the $r$-Whitney numbers of the second kind \cite{Mez3}. Equation \eqref{DowlingSpivey1} was proved using the classical operators $X$ and $D$ satisfying  the commutation relation 
\begin{equation*}
[D,x]:=DX-XD=1.
\end{equation*}
Inspecting equations \eqref{spiveyr1} and \eqref{spiveyr2}, we see that generalizing Spivey's formula yields two forms. In the first form, as seen in the right-hand side of \eqref{spiveyr1}, the $r$-Bell numbers $B_{\ell+n,r}$ are expressed recursively in terms of the $r$-Bell numbers $B_{i,r}$. In the second form, as seen in \eqref{spiveyr2}, the right-hand side involves the usual Bell numbers $B_i$ instead of $B_{i,r}$. We also notice the presence of $(k+r)^{n-i}$ instead of $k^{n-i}$.

In this paper, we will extend the $r$-Dowling polynomials to the bivariate case and investigate generalizations of Spivey's formula that are analogous to the two forms mentioned above.

\section{Bivariate $r$-Dowling polynomials}

The $r$-Whitney numbers of the second kind are defined as coefficients in the expansion of the horizontal generating function \cite{Mez3}
\begin{equation}
(mt+r)^n=\sum_{k=0}^nm^kW_{m,r}(n,k)(t)_k,
\end{equation}
and have the exponential generating function
\begin{equation}
\sum_{n=k}^{\infty}W_{m,r}(n,k)\frac{t^n}{n!}=\frac{e^{rt}}{k!}\left(\frac{e^{mt}-1}{m}\right)^k\label{rw1}
\end{equation}
and the explicit formula
\begin{equation}
W_{m,r}(n,k)=\frac{1}{m^kk!}\sum_{j=0}^k(-1)^{k-j}\binom{k}{j}(mj+r)^n.\label{rw2}
\end{equation}
Apparently, other mathematicians worked on numbers which are equivalent to $W_{m,r}(n,k)$. More precisely, the $(r,\beta)$-Stirling numbers \cite{Cor1} $\rbstirling{n}{k}_{r,\beta}$ defined by
\begin{equation*}
t^n=\sum_{k=0}^n\binom{\frac{t-r}{\beta}}{k}\beta^kk!\rbstirling{n}{k}_{r,\beta},
\end{equation*}
the Ruci\'{n}ski-Voigt numbers \cite{Rucin} $S^{n}_{k}(\a)$ defined by
\begin{equation*}
t^n=\sum_{k=0}^nS^{n}_{k}(\a)P^{\a}_{k}(x),
\end{equation*}
where $\a=(a,a+r,a+2r,a+3r,\ldots)$ and $P^{\a}_{k}(x)=\prod_{i=0}^{k-1}(t-a+ir)$, and the noncentral Whitney numbers of the second kind \cite{Mahid2} defined by
\begin{equation*}
\widetilde{W}_{m,a}(n,k)=\frac{1}{m^kk!}\left[\Delta^k(mt-a)^n\right]_{t=0},
\end{equation*}
can be expressed as
\begin{equation*}
\rbstirling{n}{k}_{r,\beta}=W_{\beta,r}(n,k),\ S^{n}_{k}(\a)=W_{r,a}(n,k),\ \widetilde{W}_{m,a}(n,k)=W_{m,-a}(n,k).
\end{equation*}
Furthermore, aside from the classical Stirling numbers of the second kind which are given by $W_{1,0}(n,k)=\sstirling{n}{k}$, the numbers considered by previous authors in \cite{Bro,Koutras,Benoumhani,Bel,Mahid3,Mahid4} can also be obtained from $W_{m,r}(n,k)$ by assigning suitable values to the parameters $m$ and $r$.

Now, looking at the defining relations in equations \eqref{b1.2} and \eqref{b4}, it is natural to define the bivariate $r$-Dowling polynomials by
\begin{equation}
D_{m,r}(n;x,y)=\sum_{k=0}^nW_{m,r}(n,k)(x)_ky^k.\label{xydowling1}
\end{equation}
In the following theorems, we will present some combinatorial properties of $D_{m,r}(n;x,y)$:

\begin{theorem}
The bivariate $r$-Dowling polynomials satisfy the following exponential generating function:
\begin{equation}
\sum_{n=0}^{\infty}D_{m,r}(n;x,y)\frac{t^n}{n!}=e^{rt}\left[1+\frac{y(e^{mt}-1)}{m}\right]^x.\label{res1}
\end{equation}
\end{theorem}
\begin{proof}
Making use of the exponential generating function in \eqref{rw1} and the binomial theorem, we get
\begin{eqnarray*}
\sum_{n=0}^{\infty}D_{m,r}(n;x,y)\frac{t^n}{n!}&=&\sum_{n=0}^{\infty}\left[\sum_{k=0}^nW_{m,r}(n,k)(x)_ky^k\right]\frac{t^n}{n!}\\
&=&\sum_{k=0}^{\infty}(x)_ky^k\frac{e^{rt}}{k!}\left(\frac{e^{mt}-1}{m}\right)^k\\
&=&e^{rt}\sum_{k=0}^x\binom{x}{k}\left[\frac{y(e^{mt}-1)}{m}\right]^k\\
&=&e^{rt}\left[1+\frac{y(e^{mt}-1)}{m}\right]^x
\end{eqnarray*}
as desired.
\end{proof}
The results in \eqref{b3} and in \cite[Theorem 3]{Zheng} are special cases of this theorem, i.e. when $m=1$ and $r=1$, and $m=1$, respectively.
\begin{theorem}
The bivariate $r$-Dowling polynomials satisfy the following explicit formula:
\begin{equation}
D_{m,r}(n;x,y)=\sum_{i=0}^x\binom{x}{i}(mi+r)^n\left(\frac{y}{m}\right)^i\left(1-\frac{y}{m}\right)^{x-i}.\label{res2}
\end{equation}
\end{theorem}
\begin{proof}
By applying the explicit formula in \eqref{rw2},
\begin{eqnarray*}
D_{m,r}(n;x,y)&=&\sum_{k=0}^n\left[\frac{1}{m^kk!}\sum_{j=0}^k(-1)^j\binom{k}{j}(m(k-j)+r)^n\right](x)_ky^k\\
&=&\sum_{j=0}^{\infty}\sum_{k=j}^{\infty}\frac{(-1)^j(m(k-j)+r)^n(x)_ky^k}{m^kj!(k-j)!}.
\end{eqnarray*}
Letting $i=k-j$ and since $(x)_{i+j}=(x)_i(x-i)_j$,
\begin{eqnarray*}
D_{m,r}(n;x,y)&=&\sum_{i=0}^{\infty}\frac{(mi+r)^n(x)_i}{i!}\left(\frac{y}{m}\right)^i\sum_{j=0}^{\infty}\frac{(-y)^j(x-i)_j}{m^jj!}\\
&=&\sum_{i=0}^{\infty}\binom{x}{i}\left(\frac{y}{m}\right)^i(mi+r)^n\sum_{j=0}^{\infty}\binom{x-i}{j}\left(\frac{-y}{m}\right)^i
\end{eqnarray*}
which simplifies into \eqref{res2}.
\end{proof}
Similar formulas for the bivariate Bell and $r$-Bell polynomials can be obtained directly from this theorem.
\begin{corollary}
The bivariate Bell and $r$-Bell polynomials satisfy the following explicit formulas:
\begin{equation}
D_{1,r}(n;x,y):=B_{n,r}(x,y)=\sum_{i=0}^x\binom{x}{i}(i+r)^ny^i(1-y)^{x-i}
\end{equation}
\begin{equation}
D_{1,0}(n;x,y):=B_n(x,y)=\sum_{i=0}^x\binom{x}{i}i^ny^i(1-y)^{x-i}.
\end{equation}
\end{corollary}
Before proceeding, we first cite the binomial inversion formula given by
\begin{equation*}
f_n=\sum_{j=0}^n\binom{n}{j}g_j\Longleftrightarrow g_n=\sum_{j=0}^n(-1)^{n-j}\binom{n}{j}f_j.
\end{equation*}
This identity will be used in the proof of the next theorem.
\begin{theorem}
The bivariate $r$-Dowling polynomials satisfy the following recurrence relations:
\begin{equation}
D_{m,r+1}(n;x,y)=\sum_{j=0}^n\binom{n}{j}D_{m,r}(j;x,y)\label{res3.1}
\end{equation}
\begin{equation}
D_{m,r}(n;x,y)=\sum_{j=0}^n(-1)^{n-j}\binom{n}{j}D_{m,r+1}(j;x,y).\label{res3.2}
\end{equation}
\end{theorem}
\begin{proof}
From \cite[Corollary 3.5]{Cheon}, the $r$-Whitney numbers of the second kind satisfy the vertical recurrence relation
\begin{equation*}
W_{m,r+1}(n,k)=\sum_{j=k}^n\binom{n}{j}W_{m,r}(j,k).
\end{equation*}
Multiplying both sides by $(x)_ky^k$ and summing over $k$ yields
\begin{equation*}
\sum_{k=0}^nW_{m,r+1}(n,k)(x)_ky^k=\sum_{j=0}^n\binom{n}{j}\sum_{k=0}^jW_{m,r}(j,k)(x)_ky^k.
\end{equation*}
Thus, by \eqref{xydowling1}, we get \eqref{res3.1}. Moreover, with $f_n=D_{m,r+1}(n;x,y)$ and $g_j=D_{m,r}(j;x,y)$, \eqref{res3.2} is obtained by using the binomial inversion formula. This completes the proof.
\end{proof}
The next corollary is obvious.
\begin{corollary}
The bivariate $r$-Bell polynomials satisfy the following recurrence relations:
\begin{equation}
B_{n,r+1}(x,y)=\sum_{j=0}^n\binom{n}{j}B_{j,r}(x,y)\label{res3.3}
\end{equation}
\begin{equation}
B_{n,r}(x,y)=\sum_{j=0}^n(-1)^{n-j}\binom{n}{j}B_{j,r+1}(x,y).\label{res3.4}
\end{equation}
\end{corollary}

Mez\H{o} \cite[Theorem 3.2]{Mez1} established the ordinary generating function of the $r$-Bell polynomials as
\begin{equation}
\sum_{n=0}^{\infty}B_{n,r}(x)t^n=\frac{-1}{rt-1}\cdot\frac{1}{e^x}\cdot{_1}{F}{_1}\left(\left.\begin{tabular}{llll}$\;\,\frac{rt-1}{t}$\\$\frac{rt+t-1}{t}$\end{tabular}\right|x\right),
\end{equation}
where
\begin{equation}
{_p}{F}{_q}\left(\left.\begin{tabular}{llll}$a_1,$&$a_2,$&$\dots,$&$a_p$\\$b_1,$&$b_2$,&$\dots,$&$b_q$\end{tabular}\right|t\right)=\sum_{k=0}^\infty\frac{\langle a_1\rangle_k\langle a_2\rangle_k\cdots\langle a_p\rangle_k}{\langle b_1\rangle_k\langle b_2\rangle_k\cdots\langle b_q\rangle_k}\frac{t^k}{k!},\label{hyper}
\end{equation}
is the hypergeometric function. This formula was then generalized in the papers of Corcino and Corcino \cite[Theorem 4.1]{Cor2} and Mangontarum et al. \cite[Theorem 39]{Mahid2}. Since the generating function in \cite[pp. 2339]{Cheon} can be written as
\begin{equation*}
\sum_{n=0}^{\infty}W_{m,r}(n,k)t^n=\frac{1}{m^k(1-rt)}\cdot\frac{(-1)^k}{\langle \frac{(m+r)t-1}{mt}\rangle_k},
\end{equation*}
then
\begin{equation*}
\sum_{k=0}^{\infty}\left(\sum_{n=0}^{\infty}W_{m,r}(n,k)t^n\right)(x)_ky^k=\frac{1}{1-rt}\sum_{k=0}^{\infty}\frac{\langle -x\rangle_k\langle 1\rangle_k}{\langle \frac{(m+r)t-1}{mt}\rangle_k}\left(\frac{y}{m}\right)^k\frac{1}{k!}.
\end{equation*}
By \eqref{hyper},
\begin{equation*}
\sum_{n=0}^{\infty}\left(\sum_{k=0}^{\infty}W_{m,r}(n,k)(x)_ky^k\right)t^n=\frac{1}{1-rt}\cdot{_2}{F}{_1}\left(\left.\begin{tabular}{llll}$-x,\;\, 1$\\$\frac{(m+r)t-1}{mt}$\end{tabular}\right|\frac{y}{m}\right)
\end{equation*}
and the next theorem follows by applying the formula \cite[pp. 559]{Abra}
\begin{equation*}
(1-t)^{-b}{_2}{F}{_1}\left(\left.\begin{tabular}{llll}$b,\;\, c-a$\\$\;\;\;\,c$\end{tabular}\right|\frac{t}{t-1}\right)={_2}{F}{_1}\left(\left.\begin{tabular}{llll}$a,\;\, b$\\$\;\;\,c$\end{tabular}\right|t\right)
\end{equation*}
with $c=\frac{(m+r)t-1}{mt}$, $a=\frac{rt-1}{mt}$ and $b=-x$.
\begin{theorem}
The bivariate $r$-Dowling polynomials have the following ordinary generating function:
\begin{equation}
\sum_{n=0}^nD_{m,r}(n;x,y)t^k=\frac{1}{1-rt}\left(\frac{m-y}{m}\right)^x{_2}{F}{_1}\left(\left.\begin{tabular}{llll}$\frac{rt-1}{mt},\;\, -x$\\$\;\,\frac{(m+r)t-1}{mt}$\end{tabular}\right|\frac{y}{y-m}\right).
\end{equation}
\end{theorem}
This yields similar generating functions for $B_n(x,y)$ and $B_{n,r}(x,y)$.
\begin{corollary}
The bivariate Bell and $r$-bell polynomials have the following ordinary generating functions:
\begin{equation}
\sum_{n=0}^nB_{n,r}(x,y)t^k=\frac{(1-y)^x}{1-rt}{_2}{F}{_1}\left(\left.\begin{tabular}{llll}$\frac{rt-1}{t},\;\, -x$\\$\;\,\frac{(1+r)t-1}{t}$\end{tabular}\right|\frac{y}{y-1}\right)
\end{equation}
\begin{equation}
\sum_{n=0}^nB_n(x,y)t^k=(1-y)^x{_2}{F}{_1}\left(\left.\begin{tabular}{llll}$-\frac{1}{t},\;\, -x$\\$\;\;\;\,\frac{t-1}{t}$\end{tabular}\right|\frac{y}{y-1}\right).
\end{equation}
\end{corollary}
Let $v_k$, $k=0,1,2,\ldots$, be a sequence of real numbers. $v_k$ is convex \cite[pp. 268]{Comt} on an interval $[a,b]$, where $[a,b]$ contains at least three consecutive integers, if
\begin{equation*}
v_k\leq \frac{1}{2}\left(v_{k-1}+v_{k+1}\right),\ k\in[a+1,b-1].
\end{equation*}
This is called convexity property. Corcino and Corcino \cite[Theorem 2.2]{Cor2} showed that Hsu and Shiue's \cite{Hsu} generalized exponential polynomials obey the convexity property. Special cases can also be seen in \cite[Theorem 42]{Mahid2} and \cite[Theorem 9]{Mahid4}.
\begin{theorem}
The bivariate $r$-Dowling polynomials satisfy the convexity property.
\end{theorem}
\begin{proof}
Let $mi+r\geq 0$ so that $(mi+r)^2\geq 0$ and 
$$0\leq1-2(mi+r)+(mi+r)^2.$$
Rewrite this as 
$$mi+r\leq\frac{1}{2}\left[1+(mi+r)^2\right]$$
and multiply $(mi+r)^n$ to both sides to get 
$$(mi+r)^{n+1}\leq\frac{1}{2}\left[(mi+r)^n+(mi+r)^{n+2}\right].$$
Multiplying both sides of this inequality by $\binom{x}{i}\left(\frac{y}{m}\right)^i\left(1-\frac{y}{m}\right)^{x-i}$, summing over $i$ and using \eqref{res2} gives
$$D_{m,r}(n+1;x,y)\leq\frac{1}{2}\left[D_{m,r}(n;x,y)+D_{m,r}(n+2;x,y)\right]$$
which is the desired result.
\end{proof}
\begin{remark}\rm
By assigning suitable values to $m$ and $r$, it can be shown that convexity property is preserved for the cases of both the bivariate Bell and $r$-Bell polynomials.
\end{remark}

\section{Generalized Spivey's formula}

Let $f(x)$ be the exponential generating function of the sequence $\{A_n\}$ given by
\begin{equation*}
\sum_{n=0}^{\infty}A_n\frac{x^n}{n!}=f(x).
\end{equation*}
The exponential generating function of the sequence $\{A_{i+j}\}$ is given by
\begin{equation}
\sum_{i=0}^{\infty}\sum_{j=0}^{\infty}A_{i+j}\frac{x^i}{i!}\frac{y^j}{j!}=f(x+y).\label{dex1}
\end{equation}
Zheng and Li \cite[Equations 7 and 8]{Zheng} used this identity in the derivation of their main results. Adopting the same method they employed in their paper, we present the following theorems:
\begin{theorem}[Generalized Spivey's formula, first form]
The following formulas hold:
\begin{equation}
D_{m,r}(\ell+n;x,y)=\sum_{k=0}^{\ell}\sum_{i=0}^n(mk)^{n-i}\binom{n}{i}W_{m,r}(\ell,k)D_{m,r}(i;x-k,y)(x)_ky^k\label{res4.1}
\end{equation}
\begin{equation}
D_{m,r}(\ell+n)=\sum_{k=0}^{\ell}\sum_{i=0}^n(mk)^{n-i}\binom{n}{i}W_{m,r}(\ell,k)D_{m,r}(i).\label{res4.2}
\end{equation}
\end{theorem}
\begin{proof}
According equations \eqref{dex1} and \eqref{res1}, we may write
\begin{equation}
\sum_{\ell=0}^{\infty}\sum_{n=0}^{\infty}D_{m,r}(\ell+n;x,y)\frac{u^{\ell}}{\ell!}\frac{v^n}{n!}=e^{r(u+v)}\left[1+\frac{y(e^{m(u+v)}-1)}{m}\right]^x\label{prf1}
\end{equation}
In the right-hand side,
\begin{equation*}
\left[1+\frac{y(e^{m(u+v)}-1)}{m}\right]^x=\left[1+\frac{y(e^{mv}-1)}{m}+\frac{ye^{mv}(e^{mu}-1)}{m}\right]^x.
\end{equation*}
Hence, by the binomial theorem,
\begin{equation*}
e^{r(u+v)}\left[1+\frac{y(e^{m(u+v)}-1)}{m}\right]^x=e^{r(u+v)}\sum_{k=0}^{\infty}\binom{x}{k}\left[1+\frac{y(e^{mv}-1)}{m}\right]^{x-k}\left[\frac{ye^{mv}(e^{mu}-1)}{m}\right]^k.
\end{equation*}
Again, we apply \eqref{res1} to get
\begin{equation*}
e^{r(u+v)}\left[1+\frac{y(e^{m(u+v)}-1)}{m}\right]^x=\sum_{k=0}^{\infty}(x)_ky^k\frac{e^{ru}(e^{mu}-1)^k}{k!m^k}\sum_{i=0}^{\infty}D_{m,r}(i;x-k,y)\frac{v^i}{i!}\sum_{j=0}^{\infty}\frac{(mvk)^j}{j!}.
\end{equation*}
From \eqref{rw1},
\begin{equation*}
e^{r(u+v)}\left[1+\frac{y(e^{m(u+v)}-1)}{m}\right]^x=\sum_{k=0}^{\infty}\sum_{\ell=k}^{\infty}\sum_{i=0}^{\infty}\sum_{j=0}^{\infty}(x)_ky^kW_{m,r}(n,k)D_{m,r}(i;x-k,y)\frac{u^{\ell}v^{i+j}(mk)^j}{\ell !i!j!}.
\end{equation*}
Reindexing the sums with $i+j=n$, and after a few simplifications 
\begin{eqnarray*}
e^{r(u+v)}\left[1+\frac{y(e^{m(u+v)}-1)}{m}\right]^x&=&\sum_{\ell=0}^{\infty}\left\{\sum_{n=0}^{\infty}\left\{\sum_{k=0}^{\ell}\sum_{i=0}^n(x)_ky^kW_{m,r}(n,k)\right.\right.\\
& &\times \left.\left.D_{m,r}(i;x-k,y)(mk)^{n-i}\binom{n}{i}\right\}\frac{v^n}{n!}\right\}\frac{u^{\ell}}{\ell !}.
\end{eqnarray*}
We arrive at the desired result in \eqref{res4.1} by combining the last equation with \eqref{prf1} and comparing the coefficients of $\frac{v^n}{n!}\cdot\frac{u^{\ell}}{\ell !}$. For \eqref{res4.2}, we simply replace $y$ with $1/x$ and then take the limit as $x\rightarrow\infty$.
\end{proof}

Now, if we make use of the exponential generating function of $B_n(x,y)$ in \eqref{b3} instead on \eqref{res1}, then after applying \eqref{rw1}, \eqref{prf1} becomes
\begin{eqnarray*}
e^{r(u+v)}\left[1+\frac{y(e^{m(u+v)}-1)}{m}\right]^x&=&\sum_{k=0}^{\infty}\sum_{\ell=k}^{\infty}\sum_{i=0}^{\infty}\sum_{j=0}^{\infty}(x)_ky^kW_{m,r}(n,k)\\
& &\times B_i\left(x-k,\frac{y}{m}\right)\frac{u^{\ell}m^iv^{i+j}(mk+r)^j}{\ell !i!j!}.
\end{eqnarray*}
Therefore, we can directly deduce the following from the previous theorem:
\begin{theorem}[Generalized Spivey's formula, second form]
The following formulas hold:
\begin{equation}
D_{m,r}(\ell+n;x,y)=\sum_{k=0}^{\ell}\sum_{i=0}^n(mk+r)^{n-i}\binom{n}{i}W_{m,r}(\ell,k)B_i\left(x-k,\frac{y}{m}\right)(x)_ky^k\label{res5.1}
\end{equation}
\begin{equation}
D_{m,r}(\ell+n)=\sum_{k=0}^{\ell}\sum_{i=0}^n(mk+r)^{n-i}\binom{n}{i}W_{m,r}(\ell,k)B_i\left(\frac{1}{m}\right).\label{res5.2}
\end{equation}
\end{theorem}

\section{Conclusion}

It is easy to see that when $m=1$, we recover from equations \eqref{res4.1} and \eqref{res5.1} Zheng and Li's \cite{Zheng} identities in \eqref{xyspiveyr1} and \eqref{xyspiveyr2}, respectively. On the other hand, equations \eqref{res4.2} and \eqref{res5.2} are both generalizations of Spivey's formula since the two equations reduce to \eqref{spivey} when $m=1$ and $r=0$. Finally, observe that when $n=0$ in \eqref{res4.2}, we get
\begin{equation*}
D_{m,r}(\ell;x,y)=\sum_{k=0}^{\ell}W_{m,r}(n,k)(x)_ky^k,
\end{equation*}
exactly the defining relation in \eqref{xydowling1}; and when $\ell=1$ in the same equation, we have the recurrence relation
\begin{equation}
D_{m,r}(n+1;x,y)=\sum_{i=0}^nm^{n-i}\binom{n}{i}D_{m,r}(i;x-1,y)xy.
\end{equation}
This scenario is very similar to how Spivey's formula generalizes both equations \eqref{b1} and \eqref{b2} as mentioned earlier in this paper. When $m=1$ and $r=0$, this results to
\begin{equation}
B_{n+1}(x,y)=\sum_{i=0}^n\binom{n}{i}B_i(x-1)xy,
\end{equation}
the bivariate extension of \eqref{b2}.


\begin{thebibliography}{99}

\bibitem{Abra}
{M. Abramowitz and I. A. Stegun, eds., \textit{Handbook of Mathematical Functions with Formulas, Graphs, and Mathematical Tables} (9th printing), Dover, 1972.}

\bibitem{Bel}
{H. Belbachir and I. Bousbaa, Translated Whitney and $r$-Whitney numbers: a combinatorial approach, \textit{J. Integer Seq.} \textbf{16} (2013). \href{https://cs.uwaterloo.ca/journals/JIS/VOL16/Belbachir/belbachir32.html}{Article 13.8.6}.}

\bibitem{Belbachir}
{H. Belbachir and M. Mihoubi, A generalized recurrence for Bell polynomials: An alternate approach to Spivey and Gould--Quaintance formulas, \textit{European J. Combin.} \textbf{30} (2009), 1254--1256.}

\bibitem{Benoumhani}
{M. Benoumhani, On Whitney numbers of Dowling lattices, \textit{Discrete Math.} \textbf{159} (1996), 13--33.}

\bibitem{Khristo}
{K. N. Boyadzhiev, Exponential polynomials, Stirling numbers,
and evaluation of some gamma integrals, \textit{Abstr. Appl. Anal.}, \textbf{2009}, Article ID 168672, 18 pages, (2009).}

\bibitem{Bro}
{A. Broder, The $r$-Stirling numbers, \textit{Discrete Math.} \textbf{49} (1984), 241--259.}

\bibitem{Cheon}
{G.-S. Cheon and J.-H. Jung, The $r$-Whitney numbers of Dowling lattices, \textit{Discrete Math.}, \textbf{15} (2012), 2337--2348.}

\bibitem{Comt}
{L. Comtet, \textit{Advanced Combinatorics}, D. Reidel Publishing Co., 1974.}

\bibitem{Cor1} 
{R. B. Corcino, The $(r,\beta)$-Stirling numbers, \textit{The Mindanao Forum} \textbf{14} (1999), 91--99.}

\bibitem{Cor2}
{R. B. Corcino and C. B. Corcino, On generalized Bell polynomials, \textit{Discrete Dyn. in Nat. and Soc.}, \textbf{2011}, Article ID 623456, 21 pages, (2011).}

\bibitem{Gould}
{H. W. Gould and J. Quaintance, Implications of Spivey's Bell number formula, \textit{J. Integer Seq.} \textbf{11} (2008). \href{https://cs.uwaterloo.ca/journals/JIS/VOL11/Gould/gould35.html}{Article 08.3.7}.}

\bibitem{Hsu}
{L. Hsu and P. J. Shiue}, A unified approach to generalized Stirling numbers, \textit{Advances Appl. Math.} \textbf{20} (1998), 366--384.

\bibitem{Koutras}
{M. Koutras, Non-central Stirling numbers and some applications, \textit{Discrete Math.} \textbf{42} (1982), 73--89.}

\bibitem{Mahid3}
{M. M. Mangontarum and A. M. Dibagulun, On the translated Whitney numbers and their combinatorial properties, \textit{British Journal of Applied Science and Technology} \textbf{11} (2015), 1--15.}

\bibitem{Mahid}
{M. M. Mangontarum and A. M. Dibagulun, Some generalizations of Spivey's Bell number formula, \textit{Matimy\`as Matematika} \textbf{40} (2017), 1--12.}

\bibitem{Mahid2}
{M. M. Mangontarum, O. I. Cauntongan, and A. P. M.-Ringia, The noncentral version of the Whitney numbers: a comprehensive study, \textit{Int. J. Math. Math. Sci.} \textbf{2016}, Article ID 6206207, 16 pages, (2016).}

\bibitem{Mahid4}
{M. M. Mangontarum, A. P.-M. Ringia, and N. S. Abdulcarim, The translated Dowling polynomials and numbers, \textit{International Scholarly Research Notices} \textbf{2014}, Article ID 678408, 8 pages, (2014).}

\bibitem{Mez3}
{I. Mez\H{o}, A new formula for the Bernoulli polynomials, \textit{Results Math.}, \textbf{58} (2010), 329--335.}

\bibitem{Mez1}
{I. Mez\H{o}, The $r$-Bell numbers, \textit{J. Integer Seq.}, \textbf{14} (2011). \href{https://cs.uwaterloo.ca/journals/JIS/VOL14/Mezo/mezo9.html}{Article 11.1.1}}

\bibitem{Mez2}
{I. Mez\H{o}, The dual of Spivey's Bell number formula, \textit{J. Integer Seq.} \textbf{15} (2012). \href{https://cs.uwaterloo.ca/journals/JIS/VOL15/Mezo/mezo14.html}{Article 12.2.4}.}

\bibitem{Rucin} 
{A. Ruci\'{n}ski and B. Voigt, A local limit theorem for generalized Stirling numbers", \textit{Revue Roumaine de Math\'{e}matiques Pures et Appliqu\'{e}es} \textbf{35} (1990), 161--172.}

\bibitem{Spivey}
{M. Z. Spivey, A generalized recurrence for Bell numbers, \textit{J. Integer Seq.} \textbf{11} (2008). \href{https://cs.uwaterloo.ca/journals/JIS/VOL11/Spivey/spivey25.html}{Article 08.2.5}.}
	
\bibitem{Zheng}
{Y. Zheng and N. N. Li, Bivariate extension of Bell polynomials, \textit{J. Integer Seq.} \textbf{22} (2019). \href{https://cs.uwaterloo.ca/journals/JIS/VOL22/Li/li53.html}{Article 19.8.8}.}


\end{thebibliography}
\end{document}